\documentclass[11pt]{amsart}

\usepackage[colorlinks=true,linkcolor=blue]{hyperref}
\usepackage{amsmath}
\usepackage{amssymb}
\usepackage{amsthm}
\usepackage{bm} 

\usepackage{graphicx}
\usepackage{epstopdf}
\usepackage{epsfig}

\usepackage{tikz}
\usepackage{tikz-cd}

\usepackage{enumerate}

\theoremstyle{plain} 
\newtheorem{thm}{Theorem}[section]
\newtheorem{prop}[thm]{Proposition}
\newtheorem{lemma}[thm]{Lemma}
\newtheorem{cor}[thm]{Corollary}

\theoremstyle{remark}
\newtheorem{remark}[thm]{Remark}

\theoremstyle{definition}
\newtheorem{defin}[thm]{Definition}

\newcommand{\FF}{\mathbb{F}}

\newcommand{\NN}{\mathbb{N}}

\newcommand{\QQ}{\mathbb{Q}}

\newcommand{\ZZ}{\mathbb{Z}}

\newcommand{\Kbar}{\overline{K}}
\newcommand{\kbar}{\overline{k}}

\DeclareMathOperator{\sgn}{sgn}
\DeclareMathOperator{\Par}{Par}

\DeclareMathOperator{\Orb}{Orb}
\DeclareMathOperator{\ch}{char}
\DeclareMathOperator{\Gal}{Gal}

\DeclareMathOperator{\res}{Res}

\DeclareMathOperator{\Aut}{Aut}

\DeclareMathOperator{\charact}{char}

\newcommand{\dsps}{\displaystyle}

\newcommand{\mapa}{\bm{\alpha}}

\newcommand{\mapt}{\bm{\tau}}

\newcommand{\Pink}{\textup{Pink}}
\newcommand{\geom}{\textup{geom}}
\newcommand{\arith}{\textup{arith}}

\begin{document}

\title[Arboreal Galois groups of PCF quadratic polynomials]
{Arboreal Galois groups of postcritically finite quadratic polynomials: The periodic case}

\author[R. L. Benedetto]{Robert L. Benedetto}
\address{Robert L. Benedetto \\ Department of Mathematics and Statistics \\ Amherst College \\ Amherst, MA 01002 \\ USA}
\email{rlbenedetto@amherst.edu}

\author[D. Ghioca]{Dragos Ghioca}
\address{Dragos Ghioca \\ Department of Mathematics \\ University of British Columbia \\ Vancouver, BC V6T 1Z2 \\ Canada}
\email{dghioca@math.ubc.ca}

\author[J. Juul]{Jamie Juul}
\address{Jamie Juul \\ Department of Mathematics \\ Colorado State University \\ Fort Collins, CO 80523 \\ USA}
\email{jamie.juul@colostate.edu}

\author[T. J. Tucker]{Thomas J. Tucker}
\address{Thomas J. Tucker\\Department of Mathematics\\ University of Rochester\\
Rochester, NY, 14620 \\ USA}
\email{thomas.tucker@rochester.edu}

%

\begin{abstract} 
We provide an explicit construction of the arboreal Galois group for the postcritically finite polynomial $f(z) = z^2 +c$, where $c$ belongs to some arbitrary field of characteristic not equal to $2$.
In this first of two papers, we consider the case that the critical point is periodic.
\end{abstract}

\maketitle

\section{Introduction}

Throughout our paper, we let $K$ be a field of characteristic not equal to $2$
with algebraic closure $\Kbar$,
and let $f(z)\in K[z]$ be a polynomial of degree $2$.
After conjugating by a $K$-rational change of coordinates,
we may assume that $f(z)=z^2+c$ for some $c\in K$.


\subsection{Arboreal Galois groups}

We consider the iterates $f^n$ of $f$ under composition,
where $f^0(z):=z$, and where $f^{n+1}=f\circ f^n$ for each integer $n\geq 0$.
A point $y\in\Kbar$ is said to be \emph{fixed} if $f(y)=y$,
or \emph{periodic} if $f^n(y)=y$ for some $n\geq 1$,
or \emph{preperiodic} if $f^n(y)=f^m(y)$ for some $n>m\geq 0$.
If $y$ is periodic, then its \emph{exact period} is the smallest $n\geq 1$ for which $f^n(y)=y$.
If $y$ is preperiodic but not periodic, then we say it is \emph{strictly preperiodic}.

Given a point $x_0\in K$, then for every integer $n\geq 0$, we define
\[ K_n:=K_{x_0,n}:= K(f^{-n}(x_0))
\qquad\text{and}\qquad
G_n:= G_{x_0,n}:=\Gal(K_n/K) \]
to be the $n$-th preimage field and its associated Galois group.
Note that $\cdots K_3/K_2/K_1/K$ is a tower of field extensions,
which we view as contained in $\Kbar$.
Thus, we may further define
\[ K_\infty:=K_{x_0,\infty}:=\bigcup_{n\geq 0} K_{x_0,n}
\qquad\text{and}\qquad
G_\infty:=G_{x_0,\infty}:=\Gal(K_{x_0,\infty}/K)\cong \varprojlim_n G_{x_0,n} .\]

If the backward orbit
\[ \Orb_f^-(x_0) := \bigcup_{n\geq 0} f^{-n}(x_0) \]
contains no critical values of $f$,
then each $f^{-n}(x_0)$ has exactly $2^n$ elements. If, in addition,
$x_0$ is not periodic under $f$, then the sets $f^{-n}(x_0)$ are pairwise disjoint,
and hence $\Orb_f^-(x_0)$ has the structure of an infinite binary rooted tree $T_\infty$,
with $x\in f^{-(n+1)}(x_0)$ connected to $f(x)\in f^{-n}(x_0)$ by an edge.
Thus, the action of the Galois group $G_{\infty}$ on the backward orbit induces
an embedding of $G_{\infty}$ into the automorphism group $\Aut(T_\infty)$ of the tree.
Similarly, for each $n\geq 0$, the action of $G_n$ on $f^{-n}(x_0)$ induces an embedding
of $G_n$ into the automorphism group $\Aut(T_n)$ of the finite binary rooted tree $T_n$
with $n$ levels. For this reason, the groups $G_n$ and $G_\infty$ have come to be
known as \emph{arboreal} Galois groups.
Moreover, given our interest in this action, whenever we discuss homomorphisms
or isomorphisms between groups acting on trees, we always mean homomorphisms
that are equivariant with respect to this action. We note that the problem of fully understanding the arboreal Galois groups has generated a great deal of research in the recent years; see 
\cite{ABCCF, BDGHT, BFHJY, BGJT1, BHL17, BJ19, JKMT, Juul, PinkPCF}, for example.


\subsection{Postcritically finite quadratic polynomials}

In this paper, we consider the case that $f$ is \emph{postcritically finite}, or PCF,
meaning that all of the the critical points of $f$ are preperiodic.
Since we have assumed $f(z)=z^2+c$, the critical points are $0$ and $\infty$,
with $\infty$ necessarily fixed; thus, to say that $f$ is PCF is equivalent to saying that
$0$ is preperiodic under $f$. In this case, it is well known that $G_{\infty}$ is
of infinite index in $\Aut(T_{\infty})$.

If the critical point $0$ is preperiodic, then the values $f(0), f^2(0),\ldots,f^r(0)$ are all distinct
for some maximal integer $r\geq 1$, with $f^{r+1}(0)$ repeating one of these values.
That is, we have $f^{r+1}(0)=f^{s+1}(0)$ for some minimal integers $r>s\geq 0$.
Equivalently, since the two preimages of $f(y)$ are $\pm y$, we have $f^r(0)=-f^s(0)$
for minimal integers $r>s\geq 0$.
Note that if $s=0$, then the point $0$ is periodic,
and $r$ is the cardinality of the \emph{forward orbit}
\[ \Orb_f^+(0) := \{f^i(0) : i\geq 0 \}\]
of $0$ under $f$. Otherwise, if $s\geq 1$, then $0$ is strictly preperiodic, and $|\Orb_f^+(0)|=r+1$.
In the latter case, the point $f^{s+1}(0)=f^{r+1}(0)$ is periodic of exact period $r-s\geq 1$,
preceded by a tail $\{0,f(0),\ldots,f^s(0)\}$ of cardinality $s+1\geq 2$.


\subsection{Previous work on describing the arboreal Galois groups for PCF quadratic polynomials}

In \cite{PinkPCF}, Pink describes the group $G_{\infty}$ for each of the various
choices of $r,s$ when the quadratic polynomial $f$ is PCF, in the case that $K=\kbar(t)$
is a rational function field over an algebraically closed field $\kbar$,
and that the root point of the preimage tree is $x_0=t$.
Pink denotes this group $G^{\geom}$, 
and he proves that it is isomorphic to a subgroup of $\Aut(T_{\infty})$ that he simply calls $G$,
but which we denote $G^{\Pink}_{r,s,\infty}$.
(When $s=0$, we sometimes write simply $G^{\Pink}_{r,\infty}$.)
He defines $G^{\Pink}_{r,s,\infty}$ via explicit (topological) generators,
each arising from the action of inertia in the context of $G^{\geom}$.

When $K=k(t)$ for $k$ \emph{not} algebraically closed, Pink denotes
the resulting group $G_{\infty}$ as $G^{\arith}$, and he describes how it fits
into a short exact sequence
\[ 1 \longrightarrow G^{\Pink}_{r,s,\infty} \longrightarrow G^{\arith}
\longrightarrow \Gal(\kbar/k)/N \longrightarrow 1, \]
for some normal subgroup $N$ of $\Gal(\kbar/k)$ depending on $r$, $s$, and $k$.


\subsection{Our approach}

This paper is the first of a pair of papers in which we have two
main goals. First, for each pair of integers $r>s\geq 0$, we construct
subgroups $B_{r,s,\infty}\subseteq M_{r,s,\infty}$ of
$\Aut(T_{\infty})$, coinciding with Pink's groups
$G^{\Pink}_{r,s,\infty}\cong G^{\geom}\subseteq G^{\arith}$, and we
show that the arboreal Galois group $G_{\infty}$ is isomorphic to a
subgroup of $M_{r,s,\infty}$. Our arguments apply over general fields
with arbitrary base points, rather than restricting to the case
$K=k(t)$ with base point $t$.  Our approach to this problem is also
more concrete than that of Pink; the groups $B_{r,s,\infty}$ and
$M_{r,s,\infty}$ are defined not by generators but rather as the set
of all $\sigma\in\Aut(T_{\infty})$ satisfying certain parity
conditions, which are also used to describe how elements of $G_\infty$
act on the roots of unity contained in $K_\infty$.  One advantage of
this approach is that it can allow us to describe the intersections
$K_n \cap k(\mu_{2^\infty})$ with a great deal of precision (see
Corollary \ref{65}).

Our second goal is to present and prove necessary and sufficient conditions for $G_{\infty}$
to be the whole group $M_{r,s,\infty}$ (see Theorem~\ref{thm:condition}).  Our results generalize those of \cite{ABCCF},
which gives a similar description of $B_{2,0,\infty}$ and $M_{2,0,\infty}$.
(This is the so-called Basilica map $f(z)=z^2-1$, for which $r=2$ and $s=0$,
i.e., the critical point at $0$ is periodic of period~$2$).

This paper handles the periodic case $s=0$ for arbitrary $r\geq1$, for which we denote the above groups simply as $M_{r,\infty}$ and $B_{r,\infty}$ (see Definition~\ref{def:Mn}). We handle the strictly preperiodic cases $s>0$ in a separate paper \cite{sequel}; nevertheless, even though additional technical complications arise in the strictly preperiodic cases, the main ideas for all of our constructions already arise in the periodic case considered here.

\begin{remark}
More generally, for any $d\geq 2$, the unicritical polynomial $f(z)=z^d+c$ is PCF
if and only if the critical point at $0$ is preperiodic.
There has been growing interest in the associated arithmetic dynamics, for example in
\cite{BDGHT,Buff18,BEK19,Gok20,HT15}, including arboreal Galois investigations in \cite{BGJT1,BHL17}.
In particular, in \cite{BGJT1}, we considered the constant field extensions $K_n\cap \overline{k}$
for PCF maps $f(z)=z^p+c$ with $p$ prime,
as well as the relationship between the generic arboreal Galois groups
(when $K=k(t)$ and $x_0=t$) and the corresponding groups when specializing the parameter $t$.
It would be interesting to find explicit descriptions of the 
of the resulting arboreal Galois groups for various PCF orbit structures for $z^d+c$.
For now, however, we restrict our attention to the case $d=2$, affording us both
Pink's work and the relative simplicity of binary trees versus $d$-ary trees.
\end{remark}

The next subsection is devoted to some further notation needed to state our main results.


\subsection{Some fundamentals}
\label{ssec:label}

It will be useful to assign labels to all of the nodes of the binary rooted trees $T_n$ and $T_{\infty}$,
using the two symbols $a$, $b$ to form words.
That is, for each integer $m\geq 0$ and each node $y$ at the $m$-th level of the tree,
we assign $y$ a \emph{label} in the form of a word $w\in\{a,b\}^m$ of length $m$,
in such a way that for every such $m$ and $y$, the two nodes lying above $y$
have labels $wa,wb\in\{a,b\}^{m+1}$. (Of course, in the tree $T_n$, this latter restriction
is vacuous for nodes $y$ in the top level $m=n$.)
See Figure~\ref{fig:treelabel} for an example of a labeling of the tree $T_3$.
Although the root node has the empty label $()$, we will often denote it as $x_0$.

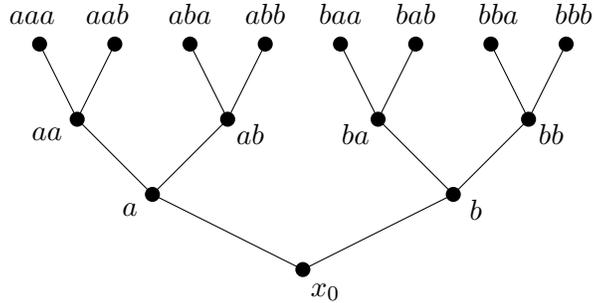
\begin{figure}
\begin{tikzpicture}
\path[draw] (.4,3.5) -- (.9,2.5) -- (1.4,3.5);
\path[draw] (2.4,3.5) -- (2.9,2.5) -- (3.4,3.5);
\path[draw] (4.4,3.5) -- (4.9,2.5) -- (5.4,3.5);
\path[draw] (6.4,3.5) -- (6.9,2.5) -- (7.4,3.5);
\path[draw] (.9,2.5) -- (1.9,1.5) -- (2.9,2.5);
\path[draw] (4.9,2.5) -- (5.9,1.5) -- (6.9,2.5);
\path[draw] (1.9,1.5) -- (3.9,0.5) -- (5.9,1.5);
\path[fill] (.4,3.5) circle (0.1);
\path[fill] (1.4,3.5) circle (0.1);
\path[fill] (2.4,3.5) circle (0.1);
\path[fill] (3.4,3.5) circle (0.1);
\path[fill] (4.4,3.5) circle (0.1);
\path[fill] (5.4,3.5) circle (0.1);
\path[fill] (6.4,3.5) circle (0.1);
\path[fill] (7.4,3.5) circle (0.1);
\path[fill] (0.9,2.5) circle (0.1);
\path[fill] (2.9,2.5) circle (0.1);
\path[fill] (4.9,2.5) circle (0.1);
\path[fill] (6.9,2.5) circle (0.1);
\path[fill] (1.9,1.5) circle (0.1);
\path[fill] (5.9,1.5) circle (0.1);
\path[fill] (3.9,0.5) circle (0.1);
\node (a) at (1.6,1.3) {$a$};
\node (b) at (6.2,1.3) {$b$};
\node (aa) at (0.5,2.3) {$aa$};
\node (ab) at (3.2,2.3) {$ab$};
\node (ba) at (4.6,2.3) {$ba$};
\node (bb) at (7.2,2.3) {$bb$};
\node (aaa) at (0.3,3.85) {$aaa$};
\node (aab) at (1.3,3.9) {$aab$};
\node (aba) at (2.4,3.9) {$aba$};
\node (abb) at (3.4,3.9) {$abb$};
\node (baa) at (4.4,3.9) {$baa$};
\node (bab) at (5.4,3.9) {$bab$};
\node (bba) at (6.5,3.9) {$bba$};
\node (bbb) at (7.5,3.9) {$bbb$};
\node (x0) at (4.2,0.2) {$x_0$};
\end{tikzpicture}
\caption{A labeling of $T_3$}
\label{fig:treelabel}
\end{figure}

We usually consider the nodes of $T_\infty$ as corresponding to the backward orbit
$\Orb_f^-(x_0)\in\Kbar$ of $x_0\in K$ under $f(z)=z^2+c\in K[z]$. Thus, we will often
conflate a point $y\in f^{-n}(x_0)$ with the corresponding node $y$ of the tree.
Having assigned a labeling to the tree, we will also sometimes conflate a node $y$
with its label.
On the other hand, when further clarity is needed for the backward orbit $\Orb_f^-(x_0)\in\Kbar$,
viewed as a tree of preimages, we will often write
the value $y\in f^{-n}(x_0)\subseteq\Kbar$ corresponding to the node
with label $w\in\{a,b\}^n$ as $y=[w]$.

Having labeled the tree, any tree automorphism
$\sigma\in\Aut(T_\infty)$ or $\sigma\in\Aut(T_n)$
must satisfy the following.
\begin{enumerate}
\item For every level $m\geq 0$ (up to $m\leq n$ for $T_n$),
$\sigma$ permutes the labels in $\{a,b\}^m$, and
\item For every level $m\geq 0$ (up to $m\leq n-1$ for $T_n$),
for each word $s_1\ldots s_m\in \{a,b\}^m$, we have either
\[ \sigma(s_1 \cdots s_m a) = \sigma(s_1 \cdots s_m) a
\quad\text{and}\quad \sigma(s_1 \cdots s_m b) = \sigma(s_1 \cdots s_m) b \]
or
\[ \sigma(s_1 \cdots s_m a) = \sigma(s_1 \cdots s_m )b
\quad\text{and}\quad \sigma(s_1 \cdots s_m) b = \sigma(s_1 \cdots s_m) a . \]
\end{enumerate}
For any tree automorphism $\sigma$ and $m$-tuple $x\in\{a,b\}^m$, we define 
the \emph{parity} $\Par(\sigma,x)$ of $\sigma$ at $x$ to be
\[ \Par(\sigma,x) := \begin{cases}
0 & \text{ if } \sigma(xa)=\sigma(x)a \text{ and } \sigma(xb)=\sigma(x)b
\\
1 & \text{ if } \sigma(xa)=\sigma(x)b \text{ and } \sigma(xb)=\sigma(x)a
\end{cases}. \]
Thus, any set of choices of $\Par(\sigma,x)$ for each node
$x$ of $T_{\infty}$ (respectively, $T_{n-1}$) determines a
unique automorphism $\sigma\in \Aut(T_{\infty})$ (respectively, $\sigma\in\Aut(T_n)$).

Note that if $\sigma(x)=x$, then $\Par(\sigma,x)$ is $0$ if $\sigma$ fixes
the two nodes above $x$, or $1$ if it transposes them. However, $\Par(\sigma,x)$
is defined even if $\sigma(x)\neq x$, although in that case its value depends also on the
labeling of the tree.

We also define $\sgn(\sigma,x)=(-1)^{\Par(\sigma,x)}$.
We have the following elementary relations:
\begin{equation}
\label{eq:sgn1}
\sgn(\sigma \tau, x) = \sgn\big(\sigma, \tau(x) \big) \cdot \sgn(\tau, x) ,
\end{equation}
and hence
\begin{equation}
\label{eq:sgn2}
\Par(\sigma\tau,x) = \Par\big(\sigma, \tau(x)\big) + \sgn\big(\sigma,\tau(x)\big) \Par(\tau,x) .
\end{equation}
Equation~\eqref{eq:sgn2} follows from equation~\eqref{eq:sgn1} by writing
$\Par(\cdot,\cdot)=(1-\sgn(\cdot,\cdot))/2$, or simply by checking the four
possible choices of $\Par(\tau,x)$ and $\Par(\sigma,\tau(x))$.

\begin{defin}
\label{def:Wri}
For each $i\geq 1$, define $W(r,i)$ to be the following set of words of length $r i-1$:
\[ W(r,i) = \big\{s_1 s_2 \cdots s_{r i-1} \, : \,
s_j\in\{a,b\},
\text{ with } s_j=a \text{ if } r | j \big\} . \]
\end{defin}

\begin{defin}
\label{def:Pmap}
Fix a labeling of $T_{\infty}$, and let $\sigma\in\Aut(T_{\infty})$.
For any word $x$ in the symbols $\{a,b\}$, define
\begin{equation}
\label{eq:Qdef}
Q_{r}(\sigma,x):= \sum_{i\geq 1} 2^i \sum_{w\in W(r,i)}
\Par(\sigma,xw) \in 2 \ZZ_2,
\end{equation}
and
\begin{equation}
\label{eq:Pdef}
P_{r}(\sigma,x):= (-1)^{\Par(\sigma,x)}  + Q_{r}(\sigma,xb) - Q_{r}(\sigma,xa)
\in \ZZ_2^{\times}.
\end{equation}
\end{defin}

Thus,
$P_{r}(\sigma,x)$ is $\pm 1$ plus a weighted sum of $\Par(\sigma,y)$
at certain nodes $y$. Specifically, the sum counts
half the nodes $r$ levels above $x$, each with weight $\pm 2$;
a quarter of the nodes $2r$ levels above $x$, each with weight $\pm 4$,
an eighth of the nodes $4r$ levels above $x$,  each with weight $\pm 8$; and so on.
(The $+$ weights are for nodes above $xb$, and the $-$ weights are for nodes above $xa$.)

\begin{defin}
\label{def:Mn}
Fix a labeling $a,b$ of $T_{\infty}$.
Define $M_{r,\infty}$ to be the subset of $\Aut(T_{\infty})$
for which 
\begin{equation}
\label{eq:Minfcond}
P_{r}(\sigma,x_1) = P_{r}(\sigma,x_2)
\quad\text{for all nodes } x_1,x_2 \text{ of } T_{\infty}.
\end{equation}
For $\sigma\in M_{r,\infty}$, define $P_r(\sigma)$ to be this common value of $P_r(\sigma,\cdot)$.
Define
\[B_{r,\infty}:=\{\sigma\in M_{r,\infty} : P_r(\sigma)=1\}.\]
\end{defin}

The map $P_r$ from $M_{r,\infty}$ to $\ZZ_2^\times$ is connected
closely to the $2$-adic cyclotomic character, as we shall see in Theorem
\ref{thm:Pembed}.  

As a final item of notation before stating our main result
(Theorem~\ref{thm:condition} in Subsection~\ref{ssec:mainthm}),
for  $0\leq m\leq n\leq \infty$, it will be convenient to define homomorphisms
\[ \res_{n,m}:\Aut(T_n)\rightarrow \Aut(T_m) \]
given by restricting elements of $\Aut(T_n)$ to the $m$-th level of the tree.
In particular, for each integer $n\geq 1$, we may
define $B_{r,n}:=\res_{\infty,n}(B_{r,\infty})$ and $M_{r,n}:=\res_{\infty,n}(M_{r,\infty})$.


\subsection{Statement of our main result}
\label{ssec:mainthm}

\begin{thm}
\label{thm:condition}
Let $K$ be a field of characteristic not equal to $2$,
and let $f(z)=z^2+c\in K[z]$ with $f^r(0)=0$
for some minimal integer $r\geq 1$.
Let $x_0\in K$, and define $K_{x_0,n}=K(f^{-n}(x_0))$, $K_{x_0,\infty}=\bigcup_{n=1}^\infty K_{x_0,n}$,
$G_{x_0,n}=\Gal(K_{x_0,n}/K)$, and $G_{x_0,\infty}=\Gal(K_{x_0,\infty}/K)$.
Further define $D_1,\ldots, D_r\in K$ by
\[ D_i := \begin{cases}
x_0-c & \text{ if } i=1, \\
f^i(0)-x_0 & \text{ if } i\geq 2.
\end{cases} \]
Then the following are equivalent.
\begin{enumerate}
\item $\dsps [K(\zeta_8,\sqrt{D_1},\ldots,\sqrt{D_r}) : K] = 2^{r+2}$
\item $[K_{x_0,2r+1}:K]=|M_{r,2r+1}|$.
\item $G_{x_0,2r+1}\cong M_{r,2r+1}$.
\item $G_{x_0,n}\cong M_{r,n}$ for all $n\geq 1$.
\item $G_{x_0,\infty}\cong M_{r,\infty}$.
\end{enumerate}
\end{thm}

\begin{remark}
  When $K$ contains $\sqrt{-1}$, the conditions of Theorem \ref{thm:condition}
  can never hold, as condition~(1) necessarily fails.
  However, our methods
  can still be used to prove a slightly more complicated result
  involving an appropriate subgroup of $M_{r,\infty}$.
  Specifically, this subgroup is the inverse image under $P_r$
  of the image in $\ZZ_2^\times$ of the $2$-adic
  cyclotomic character of $\Gal({\bar K}/K)$.
\end{remark}


\subsection{Outline of the paper}
Section~\ref{sec:notate} concerns a useful elementary result for
general quadratic polynomials that underlies many of our subsequent
arguments.  In Section~\ref{sec:roots} we present explicit formulas
yielding $2$-power roots of unity as arithmetic combinations of
preimages of an arbitrary root point under our quadratic polynomial
$f(z)=z^2+c$ when the critical point is periodic.
Section~\ref{sec:MPdef} is devoted to proving that $M_{r,\infty}$ is a
subgroup of $\Aut(T_\infty)$ and that $P_r$ is a group homomorphism with kernel  $B_{r,\infty}$ (this is done in Theorem~\ref{thm:Mgroup}), while
Section~\ref{sec:perGal} shows how $M_{r,\infty}$ and $B_{r,\infty}$
realize the arboreal Galois action.  In Section~\ref{sec:PinkGroups},
we prove that our group $B_{r,\infty}$ concides with Pink's group
$G^{\geom}$, and that Pink's larger group $G^{\arith}$ is contained in our $M_{r,\infty}$.
(Recall that $G^{\geom}$ and $G^{\arith}$ are special cases of $G_{\infty}$
for $K=\kbar(t)$ and $K=k(t)$, respectively, with $x_0=t$.)
Finally, in Section~\ref{sec:obtain}, we prove
Theorem~\ref{thm:condition}, giving necessary and sufficient
conditions for the Galois group $G_{\infty}$ to be the full group $M_{r,\infty}$.


\section{An elementary lemma}
\label{sec:notate}

The following result provides a simple but essential algebraic
relationship among elements of a backward orbit under a polynomial of
the form $f(z)=z^2+c$. We have stated it with the language of
multiplicity, but in practice we will only apply it to backward orbits
with no critical points, for which the relevant equation $f^m(z)=y$
has no repeated roots.  Note that in this lemma, we do not make any
assumptions about the polynomial $f(z)$ beyond the fact that it is of the form $z^2+c$,
whereas in later sections, we will almost always work exclusively
with quadratic polynomials satisfying $f^r(0) = 0$ for some $r \geq 1$.

\begin{prop}
\label{prop:key}
Let $K$ be a field of characteristic not equal to $2$.
Let $c\in K$, define $f(z)=z^2+c$, let $y\in\overline{K}$, and let $m\geq 1$.
Choose $\alpha_1,\ldots,\alpha_{2^{(m-1)}}\in f^{-m}(y)$ so that
the roots of $f^m(z)=y$, repeated according to multiplicity, are precisely
\begin{equation}
\label{eq:alpharoots}
\pm\alpha_1, \ldots, \pm \alpha_{2^{(m-1)}} .
\end{equation}
Then
\[ \big( \alpha_1 \alpha_2 \cdots \alpha_{2^{(m-1)}} \big)^2 =
\begin{cases}
f^m(0)-y & \text{ if } m\geq 2, \\
y-f(0) & \text{ if } m=1.
\end{cases} \]
\end{prop}

\begin{proof}
We may write $f^m(z)-y=g(z^2)$, where $g\in\overline{K}[z]$
is a polynomial of degree $2^{m-1}$.
Thus, the roots of $f^m(z)-y$ come in plus/minus pairs,
justifying the description of the roots in equation~\eqref{eq:alpharoots}.
Moreover, the roots of $g$ are precisely
$\alpha_1^2, \ldots, \alpha_{2^{m-1}}^2$,
so that the constant term of $g$ is
\begin{equation}
\label{eq:gprod}
(\pm 1)^{\deg(g)} (\alpha_1\cdots \alpha_{2^{m-1}} )^2 .
\end{equation}
Since $\deg(g)=2^{m-1}$, we have
a $-$ sign in equation~\eqref{eq:gprod} if $m=1$,
and a $+$ sign if $m\geq 2$.
On the other hand, by definition of $g$, the constant
term of $g$ is $f^m(0)-y$,
and the desired conclusion is immediate.
\end{proof}


\section{Roots of unity arising in backward orbits}
\label{sec:roots}

Throughout the rest of the paper we assume $f(z)=z^2+c$ with $f^r(0)=0$ for some minimal integer $r\geq 1$.

\begin{lemma}\label{lem:per_2nroots}
Let $x\in \Kbar$ not in the forward orbit of $0$, and let $\pm y$ be its two immediate preimages under $f$. Let $A_1=\{y\}$ and $B_1=\{-y\}$. For each $n\geq 2$, let $A_n$ be a subset of $f^{-r}(A_{n-1})$ such that $A_n$ contains exactly half of the elements of $f^{-r}(A_{n-1})$ and $f^{-r}(A_{n-1})= \{\pm \alpha :\alpha\in A_n\}$. Similarly, let $B_n$ be a subset of $f^{-r}(B_{n-1})$ containing exactly half of the elements of $f^{-r}(B_{n-1})$ such that $f^{-r}(B_{n-1})=\{\pm \beta :\beta\in B_n\}$. Then \[\gamma_n := \frac{\prod_{\alpha\in A_n}\alpha}{\prod_{\beta\in B_n}\beta}\] is a primitive $2^{n}$-th root of unity.
\end{lemma}

\begin{proof}
Let $x\in \Kbar$ not in the forward orbit of $0$ and
consider its two immediate preimages $\pm y$. First note \[\gamma_1=\frac{y}{-y}=-1,\] so the result holds in this case. 
Then by Proposition~\ref{prop:key} and the fact that $f^{r}(0)=0$, we have
\[ \gamma_2^2 = \frac{\left(\prod_{\alpha\in A_2}\alpha\right)^2}{\left(\prod_{\alpha\in B_2}\beta\right)^2}= \frac{(-1)^{2^{r-1}}(-y)}{(-1)^{2^{r-1}}y} = -1, \]
so that $\gamma_2$ is a primitive fourth root of unity. 

More generally, suppose $\gamma_{n-1}$ is a primitive $2^{n-1}$ root of unity for $n\geq 2$. For any $\alpha'\in A_{n-1}$, we have $f^{-r}(\alpha') = \{\pm\alpha: \alpha \in f^{-r}(\alpha')\cap A_{n}\}$. Hence  
\begin{align*}
\gamma_{n}^2 &= \frac{\left(\prod_{\alpha\in A_{n}}\alpha\right)^2}{\left(\prod_{\beta\in B_{n}}\beta\right)^2}
=\frac{\prod_{\alpha'\in A_{n-1}}\left(\prod_{\alpha\in A_{n}\cap f^{-r}(\alpha')}\alpha\right)^2}{\prod_{\beta'\in B_{n-1}}\left(\prod_{\beta\in B_{n}\cap f^{-r}(\beta')}\beta\right)^2}\\
&=\frac{\prod_{\alpha'\in A_{n-1}}(-1)^{2^{r-1}}(-\alpha')}{\prod_{\beta'\in B_{n-1}}(-1)^{2^{r-1}}(-\beta')}
=\frac{\prod_{\alpha'\in A_{n-1}}\alpha'}{\prod_{\beta'\in B_{n-1}}\beta'}=\gamma_{n-1},
\end{align*}
where the third equality follows from Proposition~\ref{prop:key}.
Hence $\gamma_{n}$ is a primitive $2^{n}$-th root of unity.
\end{proof}

In the statement of our next result, recall Definition~\ref{def:Wri} of the set $W(r,i)$
of words $s_1\cdots s_{ri-1}$ of length $ri-1$ with $s_j=a$ whenever $r|j$.
In addition, observe that the nodes of $T_{\infty}$ with labels $[xawa]$ for $w\in W(r,i)$
form a choice of a set $A_{i+1}$ as in Lemma~\ref{lem:per_2nroots},
and those with labels $[xbwa]$ form a choice of a set $B_{i+1}$.

\begin{lemma}\label{lem:per_labeling}
Let $x_0\in K$ not in the forward orbit of $0$, and choose a sequence of primitive $2^n$-th roots of unity
$\zeta_2,\zeta_{4}, \zeta_{8},\dots$ such that $\zeta_2=-1$ and $\zeta_{2^n}^2=\zeta_{2^{n-1}}$. It is possible to label the tree $T_\infty$ of preimages $\Orb^-_f(x_0)$ in a way such that for every node $x$ of the tree
and every integer $i\geq 1$, we have 
\begin{equation}\label{eq:perprod}
\frac{\prod_{w\in W(r,i)} [xawa]}{\prod_{w\in W(r,i)} [xbwa]}=\zeta_{2^{i+1}}.
\end{equation}
\end{lemma}

\begin{proof}
We will label the tree inductively, starting from the root point $x_0$.
Label the tree arbitrarily up to level $r+1$.

For each successive $n\geq r+1$,
suppose that we have labeled $T_{n-1}$ so that for every node $x$ at every level
$0\leq \ell\leq n-r-2$ of $T_{n-1}$,
equation~\eqref{eq:perprod} holds for each $1\leq i\leq \lfloor (n-\ell-2)/r\rfloor$.
(Note that $\lfloor (n-\ell-2)/r\rfloor$ is the maximum value of $i$ so that the subtree of height $ri+1$ rooted at $x$ is contained in $T_{n-1}$. In particular, our supposition is vacuous for $n\leq r+1$.)
For each node $y$ at level $n-1$, label the two points
of $f^{-1}(y)$ arbitrarily as $ya$ and $yb$. We will now adjust these labels that we have just
applied at the $n$-th level of the tree.

Let $m:=\lfloor (n-1)/r\rfloor \geq 1$, so that $n=rm+t$ with $1\leq t\leq r$.
Starting with $i=m$ (and counting down to $i=1$),
for each node $x$ at level $n-(ri+1)$ of the tree, consider the ratio
\[ \gamma := \frac{\prod_{w\in W(r,i)} [xawa]}{\prod_{w\in W(r,i)} [xbwa]}\]
of equation~\eqref{eq:perprod}. Arguing as in the proof of Lemma~\ref{lem:per_2nroots},
it follows from Proposition~\ref{prop:key} that
\[ \gamma^2 = \frac{\prod_{w\in W(r,i-1)} [xawa]}{\prod_{w\in W(r,i-1)} [xbwa]}
\quad\text{ if }i\geq 2,\quad\text{or}\quad
\gamma^2 = \frac{[xa]}{[xb]}=-1 \quad\text{ if }i=1, \]
which is equal to $\zeta_{2^{i-1}}$ by our induction hypothesis
when $i\geq 2$, and by definition of $\zeta_2$ when $i=1$.
Thus, $\gamma=\pm \zeta_{2^i}$. If $\gamma = -\zeta_{2^i}$,
exchange the labels of the two level-$n$ nodes $xba^{ri-1}a$ and $xba^{ri-1}b$,
where $a^j$ denotes a string of $j$ copies of the symbol $a$.
Since these two nodes are negatives of each other, we now have $\gamma = \zeta_{2^i}$. 

Repeat the process above for each $x$ at level $n-(ri+1)$ of the tree for
successively smaller $i=m-1, m-2, \dots, 1$.
Note that for any node $x$ at level $n-(ri+1)$, the nodes
$xba^{ri-1}a$ and $xba^{ri-1}b$ have a $b$ appearing as the
$(ri+1)^{\textup{st}}$-to-last-symbol in their labels.
On the other hand, for any $j>i$, by definition of $W(r,j)$, all of the nodes appearing in
the analog of equation~\eqref{eq:perprod} for $j$ in place of $i$
(and a node at level $n-(rj+1)$ in place of $x$)
have the symbol $a$ in that position in their labels.
Thus, exchanging the labels of the nodes $xba^{ri-1}a$ and $xba^{ri-1}b$ does
not affect the truth of equation~\eqref{eq:perprod} for nodes addressed in previous steps.
\end{proof}


\section{A preliminary result regarding the associated arboreal subgroup}
\label{sec:MPdef}

We now prove that the sets $B_{r,\infty}\subseteq M_{r,\infty}\subseteq\Aut(T_\infty)$
of Definition~\ref{def:Mn} are in fact groups.

\begin{thm}
\label{thm:Mgroup}
The following hold.
\begin{enumerate}
\item
$M_{r,\infty}$ is a subgroup of $\Aut(T_{\infty})$.
\item
The map $P_{r}:M_{r,\infty}\to\ZZ_2^{\times}$
given by $P_{r}:\sigma\mapsto P_{r}(\sigma)$ is a group homomorphism with kernel $B_{r,\infty}$.
\end{enumerate}
\end{thm}

\begin{proof}
\textbf{Step~1}. We begin with two simple observations that apply
to any $\tau\in\Aut(T_{\infty})$ and any node $y$ of $T_{\infty}$.
First, we have $W(r,1)=\{a,b\}^{r-1}$ is the set of all $2^{r-1}$ words of length $r-1$ in $\{a,b\}$,
and hence
\begin{equation}
\label{eq:Wr1shift}
\{ \tau(y)w : w\in W(r,1) \} = \{ \tau(yw) : w\in W(r,1) \}
\end{equation}
are precisely the same set of $2^{r-1}$ nodes of $T_{\infty}$.
Second, we have
\begin{equation}
\label{eq:Qrident}
Q_r(\tau,y) = 2 \sum_{w\in W(r,1)} \big( \Par(\tau, yw) + Q_r(\tau,ywa) \big),
\end{equation}
by definition of $Q_r$ (see equation~\eqref{eq:Qdef}), since for any $i\geq 2$, we have
\[ W(r,i) = \{ w a w' : w\in W(r,1) \text{ and } w'\in W(r,i-1) \}. \]

\medskip

\textbf{Step~2}.
For any $\sigma\in M_{r,\infty}$, any $\tau\in\Aut(T_{\infty})$, and any node $x$ of $T_{\infty}$, define
\[ Z_r(\sigma,\tau,x) := Q_r(\sigma,\tau(x)) + P_r(\sigma) Q_r(\tau,x) - Q_r(\sigma\tau,x) \in \ZZ_2 .\]
In Step~3 we will show that $Z_r$ is identically zero, but in this step we claim only that
\begin{equation}
\label{eq:Zrsum}
Z_r(\sigma,\tau,x) = 2 \sum_{w\in W(r,1)} Z_r(\sigma,\tau,xwa).
\end{equation}

To prove the claim of equation~\eqref{eq:Zrsum}, expand each appearance of $Q_r$
in the definition of $Z_r$ according to equation~\eqref{eq:Qrident}, to obtain
\begin{align*}
Z_r(\sigma,\tau,x) &= 2\sum_{w\in W(r,1)} \Big[
\Par(\sigma,\tau(x)w) + P_r(\sigma)\Par(\tau,xw) -\Par(\sigma\tau,xw) \\
& \qquad + Q_r(\sigma,\tau(x)wa) + P_r(\sigma) Q_r(\tau,xwa) - Q_r(\sigma\tau, xwa) \Big] \\
&=  2\sum_{w\in W(r,1)} \Big[
\Par(\sigma,\tau(xw)) + P_r(\sigma)\Par(\tau,xw) -\Par(\sigma\tau,xw) \\
& \qquad + Q_r(\sigma,\tau(xw)a) + P_r(\sigma) Q_r(\tau,xwa) - Q_r(\sigma\tau, xwa) \Big]
\end{align*}
by applying observation~\eqref{eq:Wr1shift} in the second equality.
Expanding the first appearance of $P_r(\sigma)$ here as $P_r(\sigma,\tau(xw))$, we have
\begin{align*}
Z_r(\sigma,\tau,x) &= 2\sum_{w\in W(r,1)} \Big[
\Par(\sigma,\tau(xw))  + (-1)^{\Par(\sigma,\tau(xw))} \Par(\tau,xw) - \Par(\sigma\tau,xw) \\
& \qquad + \Par(\tau,xw) \big( Q_r(\sigma,\tau(xw)b) - Q_r(\sigma,\tau(xw)a) \big) \\
& \qquad + Q_r(\sigma,\tau(xw)a) + P_r(\sigma) Q_r(\tau,xwa) - Q_r(\sigma\tau, xwa) \Big] \\
&= 2\sum_{w\in W(r,1)} \tilde{Z}_r(\sigma,\tau,x,w)
\end{align*}
where, after applying equation~\eqref{eq:sgn2} to $\Par(\sigma\tau,xw)$, we define
\begin{align*}
\tilde{Z}_r(\sigma,\tau,x,w) & :=
\Par(\tau,xw) \big( Q_r(\sigma,\tau(xw)b) - Q_r(\sigma,\tau(xw)a) \big) \\
& \quad + Q_r(\sigma,\tau(xw)a) + P_r(\sigma) Q_r(\tau,xwa) - Q_r(\sigma\tau, xwa).
\end{align*}

For each $w\in W(r,1)$, we consider two cases.
If $\Par(\tau,xw)=0$, then $\tau(xw)a = \tau(xwa)$, so
\[ \tilde{Z}_r(\sigma,\tau,x,w) = 0+ 
Q_r(\sigma,\tau(xwa)) + P_r(\sigma) Q_r(\tau,xwa) - Q_r(\sigma\tau, xwa)
= Z(\sigma,\tau,xwa) .\]
On the other hand, if $\Par(\tau,xw)=1$, then $\tau(xw)b = \tau(xwa)$, so
\begin{align*}
\tilde{Z}_r(\sigma,\tau,x,w) &=  Q_r(\sigma,\tau(xwa)) - Q_r(\sigma,\tau(xw)a) \\
& \quad + Q_r(\sigma,\tau(xw)a) + P_r(\sigma) Q_r(\tau,xwa) - Q_r(\sigma\tau, xwa) \\
& = Q_r(\sigma,\tau(xwa)) + P_r(\sigma) Q_r(\tau,xwa) - Q_r(\sigma\tau, xwa)
= Z(\sigma,\tau,xwa).
\end{align*}
That is, in all cases, we have $\tilde{Z}_r(\sigma,\tau,x,w) = Z(\sigma,\tau,xwa)$.
Hence,
\[ Z_r(\sigma,\tau,x) = 2\sum_{w\in W(r,1)} \tilde{Z}_r(\sigma,\tau,x,w)
= 2\sum_{w\in W(r,1)} Z_r(\sigma,\tau,xwa), \]
proving the claim of equation~\eqref{eq:Zrsum}.

\medskip

\textbf{Step~3}. We will now show that $Z_r$ is indeed identically zero,
and furthermore, with $\sigma$ and $\tau$ as in Step~2, that
$P_r(\sigma) P_r(\tau,x)=P_r(\sigma\tau,x)$.

For $\sigma,\tau,x$ as in Step~2, a straightforward induction on $i\geq 0$ gives
\[ Z_r(\sigma,\tau,x) = 2^i \sum_{w\in W(r,i)} Z_r(\sigma,\tau,xwa) \in 2^i \ZZ_2
\quad \text{for every } i\geq 0 .\]
Because $\bigcap_{i\geq 0} 2^i \ZZ_2 = \{0\}$, it follows that $Z_r(\sigma,\tau,x) = 0$.

Expanding $P_r(\tau,x)$ according to definition~\eqref{eq:Pdef}, we have
\begin{align*}
P_r(\sigma) P_r(\tau,x)
&= (-1)^{\Par(\tau,x)} P_r(\sigma) + P_r(\sigma) Q_r(\tau,xb) - P_r(\sigma) Q_r(\tau,xa) \\
&=(-1)^{\Par(\tau,x)} \big( (-1)^{\Par(\sigma,\tau(x))} + Q_r(\sigma,\tau(x)b) - Q_r(\sigma,\tau(x)a) \big) \\
& \quad + P_r(\sigma) Q_r(\tau,xb) - P_r(\sigma) Q_r(\tau,xa) \\
&=(-1)^{\Par(\sigma\tau,x)} + Q_r(\sigma,\tau(xb)) - Q_r(\sigma,\tau(xa)) \\
& \quad + P_r(\sigma) Q_r(\tau,xb) - P_r(\sigma) Q_r(\tau,xa),
\end{align*}
where in the second equality,
we expanded the first appearance of $P_r(\sigma)$ as $P_r(\sigma,\tau(x))$,
and in the third equality, we applied both equation~\eqref{eq:sgn2} and the fact that
\[ (-1)^{\Par(\tau,x)} \big( Q_r(\sigma,\tau(x)b) - Q_r(\sigma,\tau(x)a) \big)
=  Q_r(\sigma,\tau(xb)) - Q_r(\sigma,\tau(xa)) .\]
Therefore, we have
\begin{align*}
P_r(\sigma) P_r(\tau,x)
& = (-1)^{\Par(\sigma\tau,x)} + Q_r(\sigma\tau,xb) - Q_r(\sigma\tau,xa)
+ Z_r(\sigma,\tau,xb) - Z_r(\sigma,\tau,xa) \\
& = P_r(\sigma\tau,x) + 0 - 0 = P_r(\sigma\tau,x) .
\end{align*}
by definition of $P_r$ and the fact that $Z_r=0$.

\medskip

\textbf{Step~4}. 
We now show that $M_{r,\infty}$ is a subgroup of $\Aut(T_\infty)$.
The identity $e\in\Aut(T_\infty)$ clearly satisfies $Q_r(e,x)=0$ and $\Par(e,x)=0$ for all nodes $x$,
whence $P_r(e,x)=1$, so that $e\in M_{r,\infty}$.
Given $\sigma,\tau\in M_{r,\infty}$ and $x_1,x_2\in X$, we have
\[ P_r(\sigma\tau, x_1) = P_r(\sigma) P_r(\tau,x_1) = P_r(\sigma) P_r(\tau,x_2) = P_r(\sigma\tau, x_2),\]
where the first and third equalities are by Step~3, and the second is by the fact that $\tau\in M_{r,\infty}$.
Thus, $\sigma\tau\in M_{r,\infty}$. Applying Step~3 with $\sigma^{-1}$ in the role of $\tau$,
we also have
\begin{align*}
P_r(\sigma) P_r(\sigma^{-1}, x_1) &= P_r(\sigma\sigma^{-1}, x_1) = P_r(e, x_1) = P_r(e, x_2) \\
&= P_r(\sigma\sigma^{-1}, x_2) = P_r(\sigma) P(\sigma^{-1}, x_2) .
\end{align*}
Multiplying both sides on the left by $P_r(\sigma)^{-1}\in\ZZ_2^\times$, it follows that
\[ P_r(\sigma^{-1}, x_1) = P_r(\sigma^{-1}, x_2) , \]
proving that $\sigma^{-1}\in M_{r,\infty}$.
Thus, $M_{r,\infty}$ is a subgroup of $\Aut(T_\infty)$.

Finally, the map $P_r$ is a homomorphism by Step~3, and its kernel is clearly $B_{r,\infty}$.
\end{proof}


\section{The action of Galois on roots of unity}
\label{sec:perGal}
Our next result shows that $P_r(\sigma,x)$ determines the action of $\sigma\in\Aut(T_{\infty})$
on the roots of unity $\zeta_{2^n}$ constructed in Lemma~\ref{lem:per_labeling},
thus explaining the otherwise mysterious Definition~\ref{def:Pmap} of $P_r(\sigma,x)$.
In addition, a Galois element $\sigma$ must of course act consistently on roots of unity, thus explaining
the defining condition~\eqref{eq:Minfcond} of $M_{r,\infty}$, that $P_r(\sigma,x)$ be independent of $x$.
That is, for $\sigma\in \Aut(T_{\infty})$ to belong to $M_{r,\infty}$,
Definition~\ref{def:Mn} imposes only the rudimentary condition that $\sigma$ must act
consistently on roots of unity $\zeta_{2^n}$.


\begin{thm}\label{thm:Pembed}
	Let $x_0\in K$ not in the forward orbit of $0$, and choose primitive $2^n$-th roots of unity
	$\zeta_2,\zeta_4,\zeta_8,\dots\in \Kbar$ such that $\zeta_{2^n}^2=\zeta_{2^{n-1}}$. Label the tree $T_\infty$ of preimages in $\Orb_f^-(x_0)$ as in Lemma~\ref{lem:per_labeling}. Then for any node $x\in \Orb_f^-(x_0)$ and any $\sigma\in G_\infty=\Gal(K_\infty/K)$, we have 
	\begin{equation}
		\label{eq:Pembed}
		\sigma(\zeta_{2^n}) = \zeta_{2^n}^{P_{r}(\sigma,x)},
	\end{equation}
	for all $n\geq 1$.
	In particular, the image of $G_{\infty}$ in $\Aut(T_{\infty})$,
	induced by its action on $\Orb_f^-(x_0)$ via this labeling,
	is contained in $M_{r,\infty}$.
	Furthermore, if $\zeta_{2^n}\in K$ for all $n$, then this Galois image
	is contained in $B_{r,\infty}$.
\end{thm}

\begin{proof}
 If equation~\eqref{eq:Pembed} holds for $\sigma\in G_\infty$ for every node $x\in \Orb_f^-(x_0)$, and for all $n\geq 1$, then $\sigma(\zeta_{2^n})=\zeta_{2^n}^{P_r(\sigma,x)}=\zeta_{2^n}^{P_r(\sigma,x_0)}$ for all $n\geq 1$. It follows that $P_r(\sigma,x)\equiv P_r(\sigma,x_0) \mod 2^n$ for all $n\geq 1$ and hence $\sigma\in M_{r,\infty}$. Further, if $\zeta_{2^n}\in K$ for all $n$, we must have $\sigma(\zeta_{2^n})=\zeta_{2^n}^{P_r(\sigma,x)}=\zeta_{2^n}$ for all $n$, hence $P_r(\sigma,x)=1$ and $\sigma\in B_{r,\infty}$.	 
	 
Thus, it suffices to show equation~\eqref{eq:Pembed} holds for an arbitrary $\sigma\in G_\infty$,
arbitrary $x\in \Orb_f^-(x_0)$, and arbitrary $n\geq 1$.
The desired equation is trivially true for $n=1$, as $\zeta_2=-1$ and $P_r\equiv 1 \pmod{2}$.
Therefore, we may assume for the rest of the proof that $n\geq 2$.
By Lemma~\ref{lem:per_labeling}, we can write
\[ \zeta_{2^n}= \frac{\prod_{w\in W(r,n-1)} [xawa]}{\prod_{w\in W(r,n-1)} [xbwa]}, \]
and hence
\begin{equation}\label{eq:per_sigma}
\sigma(\zeta_{2^n})= \frac{\prod_{w\in W(r,n-1)} [\sigma(xawa)]}{\prod_{w\in W(r,n-1)} [\sigma(xbwa)]}.
\end{equation}

We first claim for all $i\geq 0$ and any  word $w$ (in the alphabet $\{a,b\}$), we have
\begin{align}\label{eq:perstep1}
\prod_{w'\in W(r,i)}[\sigma(wa)w'a]&=
\zeta_{2^{i+1}}^{-\Par(\sigma,w)}\prod_{w'\in W(r,i)}[\sigma(w)aw'a],
\end{align}
where for $i=0$, we interpret this equation as saying $[\sigma(wa)]=\zeta_2^{-\Par(\sigma,w)} [\sigma(w)a]$.
When $\Par(\sigma,w)=1$, we have $\sigma(wa)=\sigma(w)b$, so equation~\eqref{eq:perstep1} follows from Lemma~\ref{lem:per_labeling} applied to $\sigma(w)$. On the other hand, when $\Par(\sigma,w)=0$, we have $\sigma(wa)=\sigma(w)a$ and hence the set of words in product on the left of the equation is the same as that on the right hand side, so equation~\eqref{eq:perstep1} is vacuously true in this case.

We also observe for any words $w ,w'$, 
\begin{equation}\label{eq:perstep2}
\prod_{w''\in \{a,b\}^{r-1}}[\sigma(ww'')w']=\prod_{w''\in \{a,b\}^{r-1}}[\sigma(w)w''w'],
\end{equation} 
since the set of words in each product is the same.

Write $S_{i}:={\sum_{w\in W(r,i)}\Par(\sigma,xaw)}$. Note, any $w\in W(r,j)$ can be written as $w_1aw_2a\dots w_j$ for $w_1,\dots,w_j\in \{a,b\}^{r-1}$. Then alternately applying equation~\eqref{eq:perstep1} and equation~\eqref{eq:perstep2}, we have
\begin{align*}
\prod_{w\in W(r,n-1)} [\sigma(xawa)]&=\prod_{w_1,\dots,w_{n-1}\in\{a,b\}} [\sigma(xaw_1aw_2a\dots w_{n-1}a)]\\ 
&= \zeta_2^{-S_{n-1}}\prod_{w_1,\dots,w_{n-1}\in\{a,b\}^{r-1}} [\sigma(xaw_1a\dots w_{n-1})a]\\ \\
&= \zeta_2^{-S_{n-1}}\prod_{w_1,\dots,w_{n-1}\in\{a,b\}^{r-1}} [\sigma(xaw_1a\dots w_{n-2}a)w_{n-1}a]\\ 
&= \zeta_2^{-S_{n-1}}\zeta_{2^2}^{-S_{n-2}} \prod_{w_1,\dots,w_{n-1}\in\{a,b\}^{r-1}} [\sigma(xaw_1a\dots w_{n-2})aw_{n-1}a]\\
&\quad \vdots\\
&= \zeta_2^{-S_{n-1}}\zeta_{2^2}^{-S_{n-2}}\dots \zeta_{2^{n-1}}^{-S_{1}}\prod_{w_1,\dots,w_{n-1}\in\{a,b\}^{r-1}} [\sigma(xaw_1)aw_2a\dots w_{n-1}a]\\ 
&= \zeta_2^{-S_{n-1}}\zeta_{2^2}^{-S_{n-2}}\dots \zeta_{2^{n-1}}^{-S_{1}} \prod_{w_1,\dots,w_{n-1}\in\{a,b\}^{r-1}} [\sigma(xa)w_1aw_2a\dots w_{n-1}a]\\
\end{align*}

Using the fact that $\zeta_{2^{n-i}}=(\zeta_{2^n})^{2^i}$, we can rewrite
\[\zeta_2^{-S_{n-1}}\zeta_{2^2}^{-S_{n-2}}\cdots \zeta_{2^{n-1}}^{-S_{1}}
=\zeta_{2^n}^{- (2S_1 + 4 S_2 + \cdots + 2^{n-1} S_{n-1})}
=\zeta_{2^n}^{-Q_r(\sigma,xa)},\]
where the last equation follows from the identity
$\sum_{i=1}^{n-1}2^{i}S_{i}\equiv Q_r(\sigma,xa) \pmod{2^n}$. Thus, we have
\begin{equation}\label{eq:per_aprod}
\prod_{w\in W(r,n-1)} [\sigma(xawa)]=\zeta_{2^n}^{-Q_r(\sigma,xa)}\prod_{w\in W(r,n-1)} [\sigma(xa)wa].
\end{equation}
Similarly,
\begin{equation}\label{eq:per_bprod}
\prod_{w\in W(r,n-1)} [\sigma(xbwa)]=\zeta_{2^n}^{-Q_r(\sigma,xb)}\prod_{w\in W(r,n-1)} [\sigma(xb)wa].
\end{equation}

Plugging equation~\eqref{eq:per_aprod} and equation~\eqref{eq:per_bprod} into equation~\eqref{eq:per_sigma}, we see 
\[\sigma(\zeta_{2^n})= \zeta_{2^n}^{-Q_r(\sigma,xa)+Q_r(\sigma,xb)}\frac{\prod_{w\in W(r,n-1)} [\sigma(xa)wa]}{\prod_{w\in W(r,n-1)} [\sigma(xb)wa]}.\]
Finally, applying Lemma~\ref{lem:per_labeling} one last time, we see
\[\frac{\prod_{w\in W(r,n-1)} [\sigma(xa)wa]}{\prod_{w\in W(r,n-1)} [\sigma(xb)wa]}=\left(\frac{\prod_{w\in W(r,n-1)} [\sigma(x)awa]}{\prod_{w\in W(r,n-1)} [\sigma(x)bwa]}\right)^{(-1)^{\Par(\sigma,x)}}=\zeta_{2^n}^{(-1)^{\Par(\sigma,x)}}.\]
Thus,
\[\sigma(\zeta_{2^n}) = \zeta_{2^n}^{(-1)^{\Par(\sigma,x)}-Q_r(\sigma,xa)+Q_r(\sigma,xb)}=\zeta_{2^n}^{P_r(\sigma,x)}.
\qedhere \]
\end{proof}


\section{The arithmetic and geometric Galois groups}
\label{sec:PinkGroups}
Let $k$ be an arbitrary field not of characteristic~2.
In this section, let $K=k(t)$ where $t$ is transcendental over $k$, let $x_0=t$, let $K_\infty$ be the resulting arboreal extension of $K$, and $G_\infty = \Gal(K_\infty/K)$, the corresponding arboreal Galois group.
In addition,
let $\kbar$ be an algebraic closure of $k$, let $K':=\kbar(t)$, let $K'_{\infty}$ be the corresponding arboreal extension of $K'$ (with root point $x_0=t$), and let $G'_\infty:=\Gal(K'_{\infty}/K')$ be the corresponding arboreal Galois group for this extension.
The groups $G_\infty$ and $G'_\infty$ are called the \textit{arithmetic} and \textit{geometric} arboreal Galois groups for $f$ over $k$, and they are denoted $G^{\arith}$ and $G^{\geom}$ respectively. 
 
Note that $K'=k(t)\cdot \bar{k}$ and $K'_\infty = K_\infty \cdot \bar{k}$.
Hence, if we define $k_\infty := K_\infty \cap \bar{k}$, then the map $G^{\geom}\rightarrow \Gal(K_\infty/k_\infty(t))$ given by restricting elements to $K_\infty$ is an isomorphism. Composing this isomorphism with the injection $\Gal(K_\infty/k_\infty(t))\rightarrow G^{\arith}$ produces a natural injection $G^{\geom}\rightarrow G^{\arith}$ with cokernel isomorphic to $\Gal(k_\infty/k)$. That is, we have an exact sequence
\[ 0 \longrightarrow G^{\geom} \longrightarrow G^{\arith} \longrightarrow \Gal(k_\infty/k) \longrightarrow 0 .\]


\subsection{The geometric Galois group}
In this subsection, we show that for each $r\geq 1$, our group $B_{r,\infty}$
coincides with Pink's group $G^{\Pink}_{r,\infty}$.
Recall that in \cite{PinkPCF}, Pink proves that this group (which he denotes simply as $G$)
is isomorphic to $G^{\geom}$, for any field $k$ as above.

In \cite[Equation~(2.0.1)]{PinkPCF},
Pink describes $G^{\Pink}_{r,\infty}$ as
the closure of the subgroup of $\Aut(T_{\infty})$ generated by elements
$\mapa_1,\ldots,\mapa_r\in\Aut(T_{\infty})$ given by the recursive relations
\begin{equation}
\label{eq:PinkPerDef}
\mapa_1 = (\mapa_r, 1) \mapt,
\qquad\text{and}\qquad
\mapa_i = (\mapa_{i-1}, 1) \quad \text{for } 2\leq i \leq r .
\end{equation}
Here, $\mapt\in\Aut(T_{\infty})$
denotes the automorphism of order two swapping the subtrees based at $a$ and $b$, given by
\[ \mapt(aw) = bw \quad\text{and}\quad  \mapt(bw) = aw \]
for all  words $w\in\{a,b\}^{\NN}$. 
In addition, for any $\sigma_a, \sigma_b \in \Aut(T_{\infty})$, the automorphism
$(\sigma_a,\sigma_b)$ is the element $\sigma\in\Aut(T_\infty)$ given by
\[ \sigma(aw) = a\sigma_a(w) \quad\text{and}\quad \sigma(bw) = b\sigma_b(w) .\]

Let $w_r$ be the word of length $r$ given by $w_r:=b a^{r-1}$,
where $a^n$ denotes $n$ copies of the symbol $a$.
A straightforward induction shows that for each $i=1,\ldots, r$,
the automorphism $\mapa_i$ of equation~\eqref{eq:PinkPerDef} is given by
\begin{equation}
\label{eq:Parmapa}
\Par(\mapa_i, w) = \begin{cases}
1 & \text{ if } w = a^{i-1} w_r^n \text{ for some } n\geq 0, \\
0 & \text{ otherwise,}
\end{cases}
\end{equation}
for any word $w$,
where $w_r^n$ is the word $w_r \cdots w_r$ of length $nr$ consisting of $n$ copies of $w_r$.
For example, for $r=3$ and $i=2$, we have $\Par(\mapa_2, w)=1$ for
\[ w= a, \; abaa, \; abaabaa, \; abaabaabaa, \ldots \]
and $\Par(\mapa_2,w)=0$ otherwise.

\begin{prop}
\label{prop:PinkSubPer}
For every integer $r\geq 1$, Pink's subgroup $G^{\Pink}_{r,\infty}$ is contained in $B_{r,\infty}$.
\end{prop}

\begin{proof}
Observe that for any node $x$ of the tree, the map $\sigma\mapsto P_r(\sigma,x)$ is a continuous
function from $\Aut(T_{\infty})$ to $\ZZ_2^\times$. Indeed, if $\sigma_1,\sigma_2\in \Aut(T_{\infty})$
agree on the finite subtree extending $nr$ levels above $x$, then 
$P_r(\sigma_1,x) \equiv P_r(\sigma_2,x) \pmod{2^{n+1}}$.
It follows that $M_{r,\infty}$ and $B_{r,\infty}$ are closed subgroups of $\Aut(T_{\infty})$.
Thus, it suffices to prove that each of the generators $\mapa_1,\ldots,\mapa_r$
of Pink's group belongs to $B_{r,\infty}$.

Fix a node $x$ of the tree, and an integer $i\in \{ 1,\ldots, r\}$. We must show that $P_r(\mapa_i,x)=1$.
We consider two cases.

First, suppose that $\Par(\mapa_i,x)=1$. By equation~\eqref{eq:Parmapa},
we must have $x=a^{i-1} w_r^m$ for some integer $m\geq 0$.
Thus, the nodes $y$ above $x$ for which $\Par(\mapa_i,y)=1$ are those of the form
$y=x w_r^n$ for some $n\geq 0$.
On the other hand, according to Definition~\ref{def:Pmap},
the value of $P_r(\mapa_i,x)$ is $(-1)^{\Par(\mapa_i,x)}=-1$ plus a weighted
sum of $\Par(\mapa_i,y')$ for nodes $y'$ of the form $y'=xaw$ or $y'=xbw$
for $w\in W(r,j)$ with $j\geq 1$.
However, none of the strings $w_r^n$ begin with $a$, and all the ones with $n\geq 2$
have length greater than $r$, with $b$ rather than $a$ as their $(r+1)$-st symbol.
Thus, the only string of the form $w_r^n$ in the sets $aW(r,j)$ or $bW(r,j)$ is $w_r=ba^{r-1}$ itself. Therefore,
\[ P_r(\mapa_i,x) = (-1)^{\Par(\mapa_i,x)} + Q_r(\mapa_i,xb) - Q_r(\mapa_i,xa)
= -1 + 2 - 0 = 1.\]

Second, we are left with the case that $\Par(\mapa_i,x)=0$.
By equation~\eqref{eq:Parmapa}, $x$ is \emph{not} of the form $a^{i-1} w_r^m$,
and therefore none of the nodes at any level $nr$ above $x$ are of this form, either.
Therefore, $\Par(\mapa_i,y)=0$ for all nodes $y$ appearing in the formula for $P_r(\mapa_i,x)$
in Definition~\ref{def:Pmap}, and hence
\[ P_r(\mapa_i,x) = (-1)^{\Par(\mapa_i,x)} + Q_r(\mapa_i,xb) - Q_r(\mapa_i,xa)
= 1 + 0 - 0 = 1. \qedhere \]
\end{proof}

In fact, we wish to extend Proposition~\ref{prop:PinkSubPer}
to show that $G^{\Pink}_{r,\infty}=B_{r,\infty}$.
To this end, for each integer $n\geq 1$, define
\[ B_{r,n}:= \res_{\infty,n}( B_{r,\infty} )
\quad\text{and}\quad
G^{\Pink}_{r,n}:=\res_{\infty,n}( G^{\Pink}_{r,\infty}) , \]
as we did at the end of Subsection~\ref{ssec:label}. We must show these subgroups of $\Aut(T_n)$ coincide.

\begin{thm}
\label{thm:PinkPer}
For every integer $r\geq 1$, we have $G^{\Pink}_{r,\infty}=B_{r,\infty}$.
\end{thm}

\begin{proof}
For each integer $n\geq 1$, define
\[ B'_{r,n}:= \big\{ \sigma\in \Aut(T_n) : \forall m<n \text{ and } \forall x\in\{a,b\}^m, \,
P_r(\sigma,x)\equiv 1 \,(\textup{mod }2^{e(m,n)} ) \big\} \]
where $e(m,n):=\lfloor \frac{n-1-m}{r}\rfloor+1$.
(Here, working on the finite tree $T_n$,
we understand the sum defining $P_r(\sigma,x)$ in Definition~\ref{def:Mn}
to be truncated to include only those terms that make sense,
i.e., only those $\Par(\sigma,y)$ terms for $y$ at level $n-1$ or below.)
Observe, in light of Proposition~\ref{prop:PinkSubPer} and Definition~\ref{def:Mn}, that we have
\[ G^{\Pink}_{r,n} \subseteq B_{r,n} \subseteq B'_{r,n} . \]
Thus, it suffices to show that $|B'_{r,n}| \leq |G^{\Pink}_{r,n}|$ for all $n\geq 1$.

We proceed by induction on $n$. For $n\leq r$, we have $e(m,n)=1$ for all $m<n$,
and hence $B'_{r,n}= \Aut(T_n)$. Therefore,
\[\log_2|B'_{r,n}|=\log_2|\Aut(T_n)|= 2^n-1=\log_2|G^{\Pink}_{r,n}|,\]
where the final equality is by \cite[Proposition 2.3.1]{PinkPCF}.

Now let $n\geq r+1$, and suppose $|B'_{r,n-1}|\leq |G^{\Pink}_{r,n-1}|$. 
Let $S_{r,n}$ denote the kernel of the map $\res_{n,n-1} : B'_{r,n}\rightarrow B'_{r,n-1}$.
Then
\[ |B'_{r,n}| = |\res_{n,n-1}(B'_{r,n})|\cdot |S_{r,n}|\leq |B'_{r,n-1}|\cdot |S_{r,n}| .\]

Next, we compute the size of $S_{r,n}$.
Define $Y_n$ to be the kernel of the map $\res_{n,n-1}:\Aut(T_n)\to\Aut(T_{n-1})$;
that is, $Y_n$ is the set of $\sigma\in\Aut(T_n)$
for which $\Par(\sigma,y)=0$ for every node $y$ below level $n-1$.
Thus, $S_{r,n}$ is the set of $\sigma\in Y_n$ for which
\begin{equation}
\label{eq:condBper}
P_r(\sigma,x)\equiv 1 \pmod{2^{e(n,m)}}
\end{equation}
for all $0\leq m<n$ and all $x\in\{a,b\}^m$.

Observe, for any $\sigma\in Y_n$, any integer $0\leq m<n$, and any node $x$ at level $m$, we have
\[ P_r(\sigma,x)\equiv
1 + \sum_{i=1}^{e(n,m)-1}
2^i\sum_{w\in W(r,i)}\big[\Par(\sigma,xbw)-\Par(\sigma,xaw) \big] \pmod{2^{e(n,m)}} .\]
Note that unless $m=n-ri-1$ for some $i$, then condition~\eqref{eq:condBper}
is automatically satisfied by the assumption that $\sigma$ acts trivially on all of $T_{n-1}$. 

Furthermore, even when $m=n-ri-1$ for some $i$, then again for
$\sigma\in Y_n$, condition~\eqref{eq:condBper} reduces to
\[ P_r(\sigma,x)\equiv 1+ 2^i\sum_{w\in W(r,i)}\big[\Par(\sigma,xbw)-\Par(\sigma,xaw) \big]\equiv 1 \pmod{2^{i+1}} .\]
Thus, for any $\sigma\in Y_n$, we have $\sigma\in S_{r,n}$ if and only if
\begin{equation}
\label{eq:Jcondsum}
\sum_{w\in W(r,i)}\Par(\sigma,xbw)-\sum_{w\in W(r,i)}\Par(\sigma,xaw)
\quad \text{is even}
\end{equation}
for each $i=1,2,\ldots,\ell$, and for each node $x$ at level $n-1-ri$,
where $\ell:=\lfloor \frac{n-1}{r}\rfloor$.

To determine whether a given $\sigma\in Y_n$ belongs to $S_{r,n}$,
start with $i=\ell$ and count down to $i=1$.
For each node $x$ at level $n-1-ri$, the values of
$\Par(\sigma,xbw)$ and $\Par(\sigma,xaw)$ for $w\in W(r,i)$
can be arbitrary except for $\Par(\sigma,xba^{ri-1})$,
which must be chosen so that the sum in condition~\eqref{eq:Jcondsum} is even.
Since there are $2^{n-1-ri}$ nodes at this level, we have $2^{n-1-ri}$ parity restrictions
arising from level $i$.

Furthermore, note that parity of $\sigma$ at
$xba^{ri-1}$ did not arise in the sum in condition~\eqref{eq:Jcondsum}
for any previous values of $i$ or $x$, because of the appearance of the symbol $b$
in that particular location.
Thus, the parity restrictions noted above are all independent of one another.
With $\sigma\in Y_n$ determined by the parities $\Par(\sigma,y)$ for each
of the $2^{n-1}$ nodes $y$ at level $n-1$, and
with $2^{n-1-ri}$ such restrictions for each $i=1,2,\ldots,\ell$, it follows that
\[\log_2|S_{r,n}|=2^{n-1}-\sum_{i=1}^{\ell} 2^{n-1-ir}.\]

On the other hand, by \cite[Proposition 2.3.1]{PinkPCF}, we have
\[ \log_2 |G^{\Pink}_{r,n}|=
2^n-1-\sum_{m=0}^{n-1} 2^{n-1-m}\cdot \left\lfloor \frac{m}{r}\right\rfloor. \]
Therefore,
\begin{align*}
\log_2\big|G^{\Pink}_{r,n}\big| &- \log_2\big|G^{\Pink}_{r,n-1}\big| \\
&= \left(2^n-1-\sum_{m=0}^{n-1} 2^{n-1-m}\cdot \left\lfloor \frac{m}{r}\right\rfloor\right)-\left(2^{n-1}-1-\sum_{m=0}^{n-2} 2^{n-2-m}\cdot \left\lfloor \frac{m}{r}\right\rfloor\right)\\
&= 2^{n-1}-\sum_{m=0}^{n-1} 2^{n-1-m}\cdot \left\lfloor \frac{m}{r}\right\rfloor+\sum_{m=1}^{n-1} 2^{n-2-(m-1)}\cdot \left\lfloor \frac{m-1}{r}\right\rfloor\\
&=2^{n-1} -2^{n-1-0}\left\lfloor\frac{0}{r}\right\rfloor+\sum_{m=1}^{n-1} 2^{n-1-m}\left(\left\lfloor\frac{m-1}{r}\right\rfloor - \left\lfloor \frac{m}{r}\right\rfloor\right)\\
&=2^{n-1}-\sum_{i=1}^{\left\lfloor\frac{n-1}{r}\right\rfloor} 2^{n-1-ir}
= \log_2|S_{r,n}|.
\end{align*}
It follows that $\log_2|G^{\Pink}_{r,n-1}|+\log_2 |S_{r,n}| = \log_2|G^{\Pink}_{r,n}|$, and hence
\[ \log_2\big|B'_{r,n}\big| \leq \log_2 |B'_{r,n-1}|+\log_2 |S_{r,n}|
\leq \log_2|G^{\Pink}_{r,n-1}| + \log_2 |S_{r,n}| =\log_2 |G^{\Pink}_{r,n}|. \qedhere \]
\end{proof}


\subsection{The arithmetic Galois group}
\label{ssec:arith}

For our field $k$ not of characteristic~2,
let $\mu_{2^\infty}$ denote the set of all $2$-power roots of unity in $\kbar$.

\begin{lemma}\label{lem:k_inf} With notation as at the start of Section~\ref{sec:PinkGroups},
for  any polynomial $f(z)=z^2+c\in k[z]$ with $0$ periodic,
we have $k_\infty = k(\mu_{2^\infty})$.
\end{lemma}

\begin{proof}

By Theorem~\ref{thm:Pembed}, there is an injective homomorphism $\rho: G^{\arith} \hookrightarrow M_{r,\infty}$. Moreover, by Theorem~\ref{thm:PinkPer}, restricting this homomorphism to $\Gal(K_\infty/k_\infty(t))\cong G^{\geom}$ yields an isomorphism $G^{\geom}\cong B_{r,\infty}$.

By Lemma~\ref{lem:per_2nroots}, $k_\infty$ contains $k(\mu_{2^\infty})$.
Hence we may consider the homomorphism
\[\chi: \Gal(k_\infty/k)\rightarrow \ZZ_2^\times\]
given by the $2$-adic cyclotomic character. Moreover, by Theorem~\ref{thm:Mgroup}, we have a homomorphism
$P_r:M_{r,\infty}\to\ZZ_2^\times$ which makes the following diagram commute:
\begin{equation}
\label{eq:maindiagram}
\begin{tikzcd}
0 \arrow[r]
  & G^{\geom} \arrow[r] \arrow[d, "\wr"]
  & G^{\arith} \arrow[r] \arrow[d, hook, "\rho"]
  & \Gal(k_\infty/k) \arrow[r] \arrow[d, "\chi"]
  & 0
\\
0 \arrow[r]
  &  B_{r,\infty} \arrow[r]
  & M_{r,\infty} \arrow[r, "P_r"]
  & \ZZ_2^\times 
\end{tikzcd}
\end{equation}

Since $B_{r,\infty}$ is the kernel of $P_r$, the induced homomorphism
$\bar{P_r}: M_{r,\infty}/B_{r,\infty}\rightarrow \ZZ_2^\times$ is injective.
We also have an induced homomorphism
$\bar{\rho}:G^{\arith}/G^{\geom}\rightarrow M_{r,\infty}/B_{r,\infty}$ given by
\[ \bar{\rho}(\sigma G^{\geom})=\rho(\sigma)B_{r,\infty}. \]

We claim that $\bar{\rho}$ is also injective.
To see this, suppose $\bar{\rho}(\sigma_1 G^{\geom})=\bar{\rho}(\sigma_2 G^{\geom})$.
Then $\rho(\sigma_1)B_{r,\infty}=\rho(\sigma_2)B_{r,\infty}$,
and hence $\rho(\sigma_1\sigma_2^{-1})\in B_{r,\infty}$.
Because $\rho$ is injective and $\rho(G^{\geom})=B_{r,\infty}$,
it follows that $\rho^{-1}(B_{r,\infty})=G^{\geom}$, and hence $\sigma_1\sigma_2^{-1}\in G^{\geom}$.
Now the following diagram commutes:
\[ \begin{tikzcd}
  G^{\arith}/G^{\geom} \arrow[r, "\sim"] \arrow[d, hook, "\bar{\rho}"]
  & \Gal(k_\infty/k) \arrow[d, "\chi"]
\\
  M_{r,\infty}/B_{r,\infty} \arrow[r, "\bar{P}_r"]
  & \ZZ_2^\times 
\end{tikzcd} \]
and since $\bar{P}_r \circ \bar{\rho}$ is injective, the homomorphism $\chi$ must be injective as well.

Finally, since $\chi$ is injective and $\ker\chi = \Gal(k_\infty/k(\mu_{2^\infty}))$,
the Galois group $\Gal(k_\infty/k(\mu_{2^\infty}))$ must be trivial,
and therefore $k_\infty = k(\mu_{2^\infty})$.
\end{proof}

Note that in general, the maps $\rho$ and $\chi$ in the above proof need not be surjective. We now show for a field $k$,
we have $G^{\arith}\cong M_{r,\infty}$ if and only if $[k(\zeta_8):k]=4$, and in particular, that the number field $k=\QQ(c)$ has this property.

\begin{thm}
\label{thm:Pinklink}
Suppose $f(z)=z^2+c \in k[z]$ and $0$ is a periodic point of $f$. 
Let $K=k(t)$, let $x_0=t$,
let $K_{\infty}$ be the resulting arboreal extension of $K$,
and let $G^{\arith}:=\Gal(K_{\infty}/K)$. Then the following are equivalent:
\begin{enumerate}
\item $G^{\arith} \cong M_{r,\infty}$
\item $[k(\zeta_8):k]=4$
\item $\ch k = 0$ and $k \cap \QQ(\mu_{2^{\infty}}) = \QQ$.
\end{enumerate}
Moreover, when $k=\QQ(c)$, we have $k \cap \QQ(\mu_{2^{\infty}}) = \QQ$
and hence $G^{\arith} \cong M_{r,\infty}$.
\end{thm}

\begin{proof} 
We first prove the final statement: that for $k=\QQ(c)$, we have $k \cap \QQ(\mu_{2^{\infty}}) = \QQ$.
Since $f(z)=z^2+c$ satisfies $f^r(0)=0$,
the parameter $c$ is a root of the polynomial 
\[ \Big( \cdots \Big(\big((x^2+x)^2+x\big)^2+x\Big)^2+ \cdots +x\Big)^2 + x \in\ZZ[x], \]
which, reduced modulo $2$, is the separable polynomial
\[ x^{2^{r-1}} + x^{2^{r-2}} + \cdots + x^2 + x \in \FF_2[x]. \]
Therefore, the prime $2$ is unramified in $k=\QQ(c)$.
It follows that $k \cap \QQ(\mu_{2^{\infty}}) = \QQ$, as desired.

Before proceeding to the equivalence of statements (1)--(3), we claim that
\[ \begin{tikzcd}
0 \arrow[r]
  &  B_{r,\infty} \arrow[r]
  & M_{r,\infty} \arrow[r, "P_r"]
  & \ZZ_2^\times \arrow[r]
  & 0
\end{tikzcd} \]
is a short exact sequence. To see this, still using $k=\QQ(c)$, we have
\[\Gal(k_\infty/k)\cong \Gal(\QQ(\mu_{2^\infty})/\QQ)\cong \ZZ_2^\times,\]
and the map $\chi: \Gal(k_\infty/k)\rightarrow \ZZ_2^\times$ is an isomorphism.
Thus, the map $P_r: M_{r,\infty}\rightarrow \ZZ_2^\times$ is surjective
by the fact that the diagram~\eqref{eq:maindiagram} commutes,
and our claim above follows.

Now let $k$ be any field satisfying the hypotheses of the theorem. By the above claim,
and again by the fact that the diagram~\eqref{eq:maindiagram} commutes, we have that
\[ \begin{tikzcd}
0 \arrow[r]
  & G^{\geom} \arrow[r] \arrow[d, "\wr"]
  & G^{\arith} \arrow[r] \arrow[d, hook, "\rho"]
  & \Gal(k_\infty/k) \arrow[r] \arrow[d, "\chi"]
  & 0
\\
0 \arrow[r]
  &  B_{r,\infty} \arrow[r]
  & M_{r,\infty} \arrow[r, "P_r"]
  & \ZZ_2^\times \arrow[r]
  & 0
\end{tikzcd} \]
commutes, with both rows exact.

The implication (3)$\Rightarrow$(2) is straightforward.
To prove (2)$\Rightarrow$(3), observe that the condition
$[k(\zeta_8):k]=4$ forces $\charact k =0$, because for any prime $p$,
the root of unity $\zeta_8$ has degree at most~2 over $\FF_p$.
Furthermore, since $\QQ(\mu_{2^\infty})$ is a pro-$2$ extension of $\QQ$,
the field $k \cap \QQ(\mu_{2^{\infty}})$ is strictly larger than $\QQ$ if and only if
$k$ contains one of the three quadratic extensions of $\QQ$ contained in $\QQ(\mu_{2^{\infty}})$.
(These three fields are $\QQ(\sqrt{-1})$, $\QQ(\sqrt{2})$, and $\QQ(\sqrt{-2})$, 
all of which are contained in $\QQ(\zeta_8)$.)
However, because $[k(\zeta_8):k]=4$, the field $k$ contains none of
$\sqrt{-1}$, $\sqrt{2}$, $\sqrt{-2}$.
Hence, we have $k \cap \QQ(\mu_{2^{\infty}}) = \QQ$, as desired.

To prove (1)$\Rightarrow$(2), observe that the map $\rho: G^{\arith}\rightarrow M_{r,\infty}$
in the diagram above is an isomorphism if and only if
the cyclotomic character $\chi: \Gal(k_\infty/k)\rightarrow \ZZ_2^\times$ is surjective,
in which case statement~(2) follows immediately.

Finally, observe that statement~(3) implies 
$\Gal(k_\infty/k)\cong \Gal(\QQ(\mu_{2^\infty})/\QQ)\cong \ZZ_2^\times$,
and hence $\chi$ is surjective. Therefore, $\rho$ is also surjective
in the diagram above, from which statement~(1) follows.
\end{proof}

Theorem~\ref{thm:Pinklink} also allows us to specify when the various roots of unity $\zeta_{2^n}$
first appear in the arboreal tower, as follows.

\begin{cor}\label{65}
If $[k(\zeta_8):k]=4$,
then $K_n\cap \kbar = k(\zeta_{2^e})$ for all $n\geq 1$, where $e=\lfloor \frac{n-1}{r}\rfloor +1$.
In particular, $\zeta_{2^n}\in K_{r(n-1)+1}$, and if $n\geq 2$, then $\zeta_{2^n}\notin K_{r(n-1)}$.
\end{cor}

\begin{proof}
Let $k_n:=\kbar\cap K_n=k_\infty\cap K_n$. By Lemma \ref{lem:per_2nroots}, we have $k(\zeta_{2^e})\subseteq k_n$,
and our hypotheses imply that $[k(\zeta_{2^e}):k]=|(\ZZ/2^e\ZZ)^\times|$. Thus, it suffices to show $[k_n:k]=|(\ZZ/2^e\ZZ)^\times|$.

By Theorem \ref{thm:Pinklink}, we have $G^{\arith}\cong M_{r,\infty}$, and hence 
\[G_n^{\arith}:=\Gal(K_n/k(t))\cong M_{r,n}:= \res_{\infty,n}(M_{r,\infty}) .\]
Define $P_{r,e}:M_{r,n} \longrightarrow (\ZZ/2^e\ZZ)^\times$
by setting $P_{r,e}(\sigma):=P_r(\tau)\pmod{2^e}$ for any
$\tau\in M_{r,\infty}$ such that $\res_{\infty,n}(\tau)=\sigma$.
Observe that if $\tau_1,\tau_2\in M_{r,\infty}$ satisfy $\res_{\infty,n}(\tau_1)=\res_{\infty,n}(\tau_2)$,
then it follows from the construction of $P_r$ in
Definition~\ref{def:Pmap} that $P_r(\tau_1)\equiv P_r(\tau_2)\pmod{2^e}$.
Hence, $P_{r,e}$ is well-defined.
It is clearly a homomorphism, and according to
the proof of Theorem \ref{thm:PinkPer}, the kernel of $P_{r,e}$ is $B_{r,n}$.
Therefore, the following diagram commutes.
\[ \begin{tikzcd}
0 \arrow[r]
  &  B_{r,\infty} \arrow[r] \arrow[d, two heads, " \res_{\infty,n}"]
  &  M_{r,\infty} \arrow[r, "P_{r}"] \arrow[d, two heads, " \res_{\infty,n}"]
  & \ZZ_2^\times \arrow[r] \arrow[d, two heads, "\text{mod }2^e"]
  & 0
\\
0 \arrow[r]
  &  B_{r,n} \arrow[r]
  & M_{r,n} \arrow[r, "P_{r,e}"]
  & (\ZZ/2^e\ZZ)^\times \arrow[r]
  & 0
\end{tikzcd} \]
In particular, the bottom row is short exact, and 
\[\Gal(k_n/k)\cong G_n^{\arith}/G_n^{\geom}\cong M_{r,n}/B_{r,n}\cong (\ZZ/2^e\ZZ)^\times,\] as desired.
\end{proof}


\section{Obtaining the arboreal Galois groups}
\label{sec:obtain}

We need two more lemmas in order to prove our main result.



\begin{lemma}
\label{lem:baseext} 
Let $k$ be any field with $[k(\zeta_8):k]=4$. Suppose $f(z)=z^2+c \in k[z]$ with $0$ periodic.
Let $K=k(t)$, let $x_0=t$,
and let $K_{\infty}$ be the resulting arboreal extension of $K$. 
Let $A\in k$ such that $\sqrt{A}\in K_{\infty}$.
Then $\sqrt{A}\in k(\zeta_8)$.
\end{lemma}

\begin{proof}
Let $k_\infty:=\kbar\cap K_{\infty}$.
By Lemma~\ref{lem:k_inf}, we have $k_\infty=k(\mu_{2^{\infty}})$,
and by hypothesis we have $[k(\zeta_8):k]=4$. Thus,
\[ \Gal(k_\infty/k)\cong \ZZ_2^{\times}. \]
Since $\ZZ_2^\times$ has only three index $2$ subgroups, $k_\infty$ contains only three quadratic extensions of $k$
(formed by adjoining one of $\sqrt{-1}$, $\sqrt{2}$, or $\sqrt{-2}$),
all of which are contained in $k(\zeta_8)$; therefore, $\sqrt{A}\in k(\zeta_8)$.
\end{proof}

The following argument has been presented in various degrees of generality in \cite[Lemma 4.3]{odoni_realizing}, \cite[Lemmas 1.5 and 1.6]{Stoll}, \cite[Section 2.2]{Jones_survey}, \cite[Proposition 5.3, Theorem 6.5]{BD}. We include a proof here for completeness. 

\begin{lemma}\label{lem:disc_condition} Let $f(z)=z^2+c\in K[z]$ and let $G_n=\Gal(K(f^{-n}(x_0))/K)$. Define $D_i\in K$ by
\[ D_i := \begin{cases}
x_0-c & \text{ if } i=1, \\
f^i(0)-x_0 & \text{ if } i\geq 2.
\end{cases} \]
Then for $n\ge 1$, the following are equivalent: 
\begin{enumerate}
\item $G_n\cong \Aut(T_n)$.
\item For all $1\leq i\leq n$, $D_i\notin (K(\sqrt{D_1},\ldots,\sqrt{D_{i-1}}))^2$.
(For $i=1$, this means $D_1\notin K^2$.)
\end{enumerate}
\end{lemma}

\begin{proof} 
First suppose statement~(2) fails for some $i\leq n$.
Let $\Delta_i$ be the polynomial discriminant of $f^i(z)-x_0$.
By \cite[Proposition 3.2]{AHM}, we have
\[\Delta_i= (-1)^{2^{i-1}}2^{2^i}\big(f^i(0)-x_0\big)\Delta\big(f^{i-1}(z)-x_0\big)^2.\]
Thus, we have $\Delta_i=c_i^2D_i$ for some $c_i\in K$, so $\Delta_i$ is a square in $K\left(\sqrt{D_1},\dots, \sqrt{D_{i-1}}\right)$. 
Therefore, every $\sigma\in G_i$ acts as an even permutation in ${\rm Gal}(K_i/K_{i-1})$; hence, $G_i\not\cong \Aut(T_i)$, and so,  $G_n\not\cong \Aut(T_n)$ for all $n\geq i$. 

Conversely, assume statement~(2). For arbitrary $1\leq i\leq n$,
suppose $G_{i-1}\cong \Aut(T_{i-1})$, which is trivially true when $i=1$.
We claim that $D_i$ is not a square in $K_{i-1}$. Let
\[L_i:= K(\sqrt{D_1},\ldots,\sqrt{D_i}),\]
which is an abelian extension of $K$.
It follows from statement~(2) that $\Gal(L_i/K)\cong \{\pm 1\}^{i}$. On the other hand, we have
\[K(\sqrt{D_1},\ldots,\sqrt{D_{i-1}})\subseteq K_{i-1}.\]
In particular, if $\sqrt{D_i}\in K_{i-1}$, then we would have $L_i\subseteq K_{i-1}$.
Because $L_i/K$ is abelian, it would follow that  $\Gal(L_i/K)\subseteq G_{i-1}^{\textup{ab}}$.
That is, we would have
\[  \{\pm 1\}^i \cong \Gal(L_i/K) \subseteq G_{i-1}^{\textup{ab}}\cong
\Aut(T_{i-1})^{\textup{ab}}\cong \{\pm 1\}^{i-1},\]
since the abelianization of a wreath product $G_1\wr G_2$ is $G_1^{\textup{ab}}\times G_2^{\textup{ab}}$.
This contradiction proves our claim, that $D_i$ is not a square in $K_{i-1}$.

Armed with this claim, we now show that $[K_i:K_{i-1}]=2^{2^{i-1}}$.
(Together with our assumption that $G_{i-1}\cong \Aut(T_{i-1})$, it will then follow
that $G_i\cong \Aut(T_i)$, from which statement~(1) follows inductively.)
Let $\beta_1,\dots \beta_{2^{i-1}}$ denote the roots of $f^{i-1}(z)-x_0$,
and note that $K_i$ is formed by adjoining
\[\sqrt{\beta_1-c},\dots,\sqrt{\beta_{2^{i-1}}-c}\]
to $K_{i-1}$. By Kummer theory, the degree $[K_i:K_{i-1}]$ is the order of the group generated by the classes of the elements $\beta_j-c$ in $K_{i-1}^\times/(K_{i-1}^\times)^2$.
This order is $2^{2^{i-1}}/|V|$, where 
\[V=\Big\{(\epsilon_1,\dots,\epsilon_{2^{i-1}})\in \FF_2^{2^{i-1}} : \prod_j(\beta_j-c)^{\epsilon_j}\in (K_{i-1}^\times)^2\Big\}.\] 
Hence, to prove $[K_i:K_{i-1}]=2^{2^{i-1}}$, it suffices to show that $V=\{0\}$. 

Note that $V$ is an $\FF_2$-vector space, and that $G_{i-1}$ acts on $V$ through its action on the $\beta_j$.
That is, for any $\sigma\in G_{i-1}$, we may write $\sigma$ as a permutation in $S_{2^{i-1}}$
given by its action on the indices of the $\beta_j$.
For any $v=(\epsilon_1,\dots,\epsilon_{2^{i-1}})\in V$, we have
$\prod_j (\beta_j-c)^{\epsilon_j}\in (K_{i-1}^\times)^2$, and hence
\[\sigma\bigg(\prod_j (\beta_j-c)^{\epsilon_j}\bigg)
=\prod_j \big(\beta_{\sigma(j)}-c\big)^{\epsilon_j}\in (K_{i-1}^\times)^2.\] 
The action is given by $\sigma v= (\epsilon_{\sigma^{-1}(1)},\dots,\epsilon_{\sigma^{-1}(2^{2^{i-1}})})\in V$,
making $V$ an $\FF_2[G_{i-1}]$-module. 

%

By the orbit-stabilizer theorem, every $G_{i-1}$-orbit in $\FF_2^{2^{i-1}}$ has cardinality a power of $2$,
since this cardinality must divide $|G_{i-1}|$. Because $G_{i-1}$ acts transitively on the $\beta_j$,
there are only two singleton orbits, and hence only two orbits of any odd cardinality:
$(0,\ldots,0)$ and $(1,\ldots,1)$.

Since $V$ is an $\FF_2$-vector space, it contains $0=(0,\ldots,0)$. If $V\neq\{0\}$,
then $|V|$ is even, and hence $V$ must contain a second $G_{i-1}$-orbit of odd order,
meaning that $(1,\ldots,1)\in V$.
In that case, therefore, we have $\prod_j (\beta_j-c)\in (K_{i-1}^\times)^2$. However, 
\[\prod_j (\beta_j-c) = (-1)^{2^{i-1}}\prod_j (c-\beta_j)
=(-1)^{2^{i-1}}(f^{i-1}(c)-x_0)=(-1)^{2^{i-1}}(f^i(0)-x_0)=D_i,\]
which is a contradiction. Hence $V=\{0\}$, as desired.
\end{proof}

We are finally ready to prove Theorem~\ref{thm:condition}.

\begin{proof}[Proof of Theorem~\ref{thm:condition}.]
%
The implications (5)$\Leftrightarrow$(4)$\Rightarrow$(3)$\Leftrightarrow$(2) are trivial.
Thus, it suffices to show  (1)$\Rightarrow$(5) and (3)$\Rightarrow$(1).
Define
\[L:= K(\sqrt{D_1},\ldots,\sqrt{D_r}).\]

First assume statement~(1), that $[L(\zeta_8):K]=2^{r+2}$.
Then we have $[K(\zeta_8):K]=4$ and $[L(\zeta_8):K(\zeta_8)]=2^r$.
Consider the arboreal extension $K(t)_{t,\infty}$ of the function field $K(t)$ for the root point $t$ in place of $x_0$,
and define $K'$ to be the compositum of the degree $2$ extensions of $K$ contained in $K(t)_{t,\infty}$.
By Lemma~\ref{lem:per_2nroots}, we have $\zeta_8\in K(t)_{t,\infty}$, and hence $\zeta_8\in K'$.
Therefore, by Lemma~\ref{lem:baseext}, we in fact have $K'=K(\zeta_8)$.

Since we now have $[L(\zeta_8):K']=2^r$, Lemma~\ref{lem:disc_condition} implies that
\[\Gal(K_{x_0,r}\cdot K'/K')\cong \Aut(T_r)\cong M_{r,r}.\] 
Therefore,
\begin{equation}
\label{eq:BGJT1cond}
|\Gal(K_{x_0,r}\cdot K'/K)|=|\Gal(K_{x_0,r}\cdot K'/K')|\cdot [K':K]=|M_{r,r}|\cdot [K':K].
\end{equation}
Noting that the forward orbit of $0$ has cardinality $r$,
we now apply \cite[Theorem 4.6]{BGJT1} (whose hypotheses assume $K$ is a number field,
but only to ensure that $[K':K]$ is finite, a fact which is evident in our case).
This result says, given the length $r$ of the forward orbit of $0$,
along with condition~\eqref{eq:BGJT1cond},
that $G_{x_0,\infty}\cong M_{r,\infty}$, i.e., that statement~(5) holds.

Finally, assume statement (3), which implies that $G_{x_0,r}\cong M_{r,r} \cong \Aut(T_r)$,
and hence, by Lemma~\ref{lem:disc_condition}, that $[L : K] \geq 2^{r}$.

We claim that statement~(3) also implies $[K_{x_0,r}(\zeta_8):K_{x_0,r}]=4$.
To see this, define $H\subseteq M_{r,2r+1}$ to be the subgroup consisting of
all $\sigma\in M_{r,2r+1}$ fixing level $r$ of the tree,
and let $\tilde{H}\subseteq G_{x_0,2r+1}$ be the image of $H$ under the
assumed isomorphism $M_{r,2r+1}\cong G_{x_0,2r+1}$.
Combining Theorem~\ref{thm:Pinklink} for $k=\QQ(c)$ with Corollary~\ref{65},
the orbit of $\zeta_8$ under $H\subseteq M_{r,2r+1}$ must have cardinality~4.
(Recall that this action of $M_{r,\infty}$ on $\zeta_8$ is given by Theorem~\ref{thm:Pembed}.)
Since $\zeta_8\in K_{x_0,2r+1}$ by Lemma~\ref{lem:per_2nroots},
the orbit of $\zeta_8$ under $\tilde{H}\subseteq G_{x_0,2r+1}$ must also have cardinality~4.
But $\tilde{H}=\Gal(K_{x_0,2r+1}/K_{x_0,r})$, and hence $[K_{x_0,r}(\zeta_8):K_{x_0,r}]=4$,
proving our claim.

Because $L\subseteq K_{x_0,r}$, it follows that $[L(\zeta_8):L]\geq 4$.
Thus, recalling our observation above that $[L : K] \geq 2^{r}$, we have
\[ 2^{r+2} \geq [L(\zeta_8) : K] = [L(\zeta_8): L]\cdot [L : K] \geq 4 \cdot 2^r = 2^{r+2}, \]
proving statement (1).
\end{proof}

\medskip

\textbf{ Acknowledgments.}
The first author gratefully acknowledges the support of NSF grant DMS-2401172.
We thank the anonymous referees for their useful comments and suggestions.

\bibliographystyle{amsalpha}
\bibliography{biblio}

\end{document}